\documentclass[reqno]{amsart}

\usepackage{enumitem}
\usepackage{amssymb}
\usepackage{hyperref}

\newcommand*{\mailto}[1]{\href{mailto:#1}{\nolinkurl{#1}}}

\newtheorem{theorem}{Theorem}[section]

\newtheorem{lemma}[theorem]{Lemma}

\newtheorem{corollary}[theorem]{Corollary}

\newcommand{\R}{{\mathbb R}}

\newcommand{\C}{{\mathbb C}}

\newcommand{\cR}{{\mathcal R}}
\newcommand{\cK}{{\mathcal K}}

\newcommand{\om}{\omega}
\newcommand{\ds}{\displaystyle}
\newcommand{\erf}{\mathrm{erf}}
\newcommand{\spr}[2]{\langle #1 , #2 \rangle}
\newcommand{\I}{\mathrm{i}}
\newcommand{\E}{\mathrm{e}}

\DeclareMathOperator{\im}{Im}

\numberwithin{equation}{section}

\begin{document}

\title[Dispersion Estimates for Schr\"odinger and Klein--Gordon Equations]{Dispersion Estimates for
One-Dimensional Schr\"odinger and Klein--Gordon Equations Revisited}

\author[I. Egorova]{Iryna Egorova}
\address{B. Verkin Institute for Low Temperature Physics\\ 47, Lenin ave\\ 61103 Kharkiv\\ Ukraine}
\email{\href{mailto:iraegorova@gmail.com}{iraegorova@gmail.com}}

\author[E.\ Kopylova]{Elena Kopylova}
\address{Faculty of Mathematics\\ University of Vienna\\
Oskar-Morgenstern-Platz 1\\ 1090 Wien\\ Austria\\ and Institute for Information Transmission Problems\\ Russian Academy of Sciences\\
Moscow 127994\\ Russia}
\email{\mailto{Elena.Kopylova@univie.ac.at}}
\urladdr{\url{http://www.mat.univie.ac.at/~ek/}}

\author[V.\ A.\ Marchenko]{Vladimir Aleksandrovich Marchenko}
\address{B. Verkin Institute for Low Temperature Physics\\ 47, Lenin ave\\ 61103 Kharkiv\\ Ukraine}
\email{\href{mailto:marchenko@ilt.kharkov.ua}{marchenko@ilt.kharkov.ua}}

\author[G.\ Teschl]{Gerald Teschl}
\address{Faculty of Mathematics\\ University of Vienna\\
Oskar-Morgenstern-Platz 1\\ 1090 Wien\\ Austria\\ and International Erwin Schr\"odinger
Institute for Mathematical Physics\\ Boltzmanngasse 9\\ 1090 Wien\\ Austria}
\email{\mailto{Gerald.Teschl@univie.ac.at}}
\urladdr{\url{http://www.mat.univie.ac.at/~gerald/}}

\dedicatory{Dedicated to the memory of Boris Moiseevich Levitan}

\thanks{Russian Math. Surveys {\bf71}, 3--26 (2016)}
\thanks{{\it Research supported by the Austrian Science Fund (FWF) under Grants No.\ P27492-N25 and Y330}}


\begin{abstract}
We show that for a one-dimensional Schr\"odinger operator with a potential whose first moment is integrable
the scattering matrix is in the unital Wiener algebra of functions with integrable Fourier transforms.
Then we use this to derive dispersion estimates for solutions of the associated Schr\"odinger and Klein--Gordon equations.
In particular, we remove the additional decay conditions in the case where a resonance is present at the edge of
the continuous spectrum.
\end{abstract}

\maketitle

\section{Introduction}

We are concerned with the one-dimensional Schr\"odinger equation
\begin{equation} \label{Schr}
  \I \dot \psi(x,t)=H \psi(x,t), \quad H:=-\frac{d^2}{dx^2}  + V(x),\quad (x,t)\in\R^2,
\end{equation}
and the Klein--Gordon equation
\begin{equation} \label{KGE}
\ddot \psi(x,t)= -(H+m^2) \psi(x,t),\quad (x,t)\in\R^2,\quad m>0,
\end{equation}
with real integrable potential $V$.
In vector form equation \eqref{KGE} reads
\begin{equation} \label{KGEv}
\I\dot \Psi(t)=\mathbf{H} \Psi(t),
\end{equation}
where
\begin{equation} \label{H}
\Psi(t)=\begin{pmatrix}
  \psi(t)
  \\
  \dot\psi(t)
  \end{pmatrix},\quad
 \mathbf{H} = \I\begin{pmatrix}
  0                          &   1\\
 -H-m^2  &   0
\end{pmatrix}.
\end{equation}
More specifically, our goal is to provide dispersive decay estimates for these equations. This is a well-studied
area and one of our main contribution is a strikingly simple proof which at the same time improves previous results.
This approach depends on the fact that the scattering matrix minus the unit matrix is in the Wiener algebra (i.e.\ its Fourier transform
is integrable). Since this result is of independent interest we prove it first in Section~\ref{wa-sec}. Based
on this we will then establish our main results. To formulate them we introduce the weighted spaces
$L^p_{\sigma}=L^p_{\sigma}(\R)$, $\sigma\in\R$, associated with the norm
\begin{equation*}
   \Vert \psi\Vert_{L^p_{\sigma}}= \begin{cases} \left(\int_{\R}(1+|x|)^{p\sigma} |\psi(x)|^p dx\right)^{1/p}, & 1\le p <\infty,\\
   \sup_{x\in\R} (1+|x|)^{\sigma} |\psi(x)|, & p=\infty.
   \end{cases}
\end{equation*}
Of course, the case $\sigma=0$ corresponds to the usual $L^p$ spaces without weight.
We recall (e.g., \cite{Mar} or \cite[Sect.~9.7]{tschroe}) that for $V\in L^1_1$ the operator $H$ has a purely absolutely continuous spectrum
on $[0,\infty)$ plus a finite number of eigenvalues in $(-\infty,0)$. At the edge of the continuous spectrum there could be a resonance
if there is a corresponding bounded solution of $-\psi''+V\psi=0$
(equivalently, if the Wronskian of the two Jost solutions vanishes at this point).

The following two results hold for the Schr\"odinger equation:
\begin{theorem}\label{Main}
Let $V\in L^1_1(\R)$. Then the following decay holds
\begin{equation}\label{full}
\Vert \E^{-\I tH}P_c\Vert_{L^1\to L^\infty}=\mathcal{O}(t^{-1/2}),\quad t\to\infty.
\end{equation}
Here $P_c=P_c(H)$ is the orthogonal projection in
$L^2(\R)$ onto the continuous spectrum of $H$.
\end{theorem}
\begin{theorem}\label{Main-new}
Let $V\in L^1_2(\R)$. Then, in the non-resonant case, the following decay holds
\begin{equation}\label{full-new}
\Vert \E^{-\I tH}P_c\Vert_{L^1_1\to L^\infty_{-1}}=\mathcal{O}(t^{-3/2}),\quad t\to\infty.
\end{equation}
\end{theorem}
Note that for the free Schr\"odinger equation \eqref{Schr}
with $V=0$ the estimate \eqref{full} is immediate from the explicit formula for
the time evolution (see e.g.\ \cite[Sect.~7.3]{tschroe}).
The dispersive decay \eqref{full} for the perturbed Schr\"odinger equation
has been established by Goldberg and Schlag \cite{GS}, improving earlier results from Weder \cite{wed},
in the non-resonant case for $V\in L^1_1$ and in the resonant case
under the more restrictive condition $V\in L^1_2$ (see also \cite{DF}). We emphasize that our
approach does not require this additional decay in the resonant case. Moreover,
our proof  for Theorem~\ref{Main} is a simple application of Fubini's theorem.
To show that the extra decay in the resonant case is not needed we generalize
an old (but obviously not so well known) result from Guseinov \cite{Gus}.
We also remark that in the half-line case the analogous result for the scattering data
is well known (cf.\ Problem~3.2.1 in \cite{Mar}) and was used by Weder \cite{wed3} to prove
a corresponding result in the half-line case.

Recall that \eqref{full} has some immediate consequences:
Interpolating between unitarity of $\exp(-\I t H) : L^2 \to L^2$   
and \eqref{full} the Riesz--Thorin theorem gives
\begin{equation}\label{n1}
\Vert \E^{-\I tH}P_c\Vert_{L^{p'}\to L^p} = \mathcal{O}(t^{-1/2+1/p})
\end{equation}
for any $p\in [2,\infty]$ with $\frac{1}{p}+\frac{1}{p'}=1$. Using \eqref{n1} we can also deduce the
corresponding Strichartz estimates of \cite[Theorem~1.2]{KT}. 

The dispersive decay \eqref{full-new} has been established by Schlag \cite{schlag} in the case $V\in L^1_4$
and later refined by Goldberg \cite{gold} to the case $V\in L^1_3$. For $V\in L_2^1$ the estimate 
\eqref{full-new} as first obtained by Mizutani in \cite{M11}. Here we propose a proof of \eqref{full-new}
based on a somewhat different approach.

Note that the decay \eqref{full-new} immediately implies the following long-time asymptotics  in weighted norms:
\begin{equation}\label{m-as}
\Vert \E^{-\I tH}P_c\Vert_{L^2_\sigma\to L^2_{-\sigma}}=\mathcal{O}(t^{-3/2}),\quad t\to\infty,
\end{equation}
for any $\sigma>3/2$.
Asymptotics of type \eqref{m-as} in the non-resonant case were obtained by Murata \cite{M}
for more general (multi-dimensional) Schr\"odinger-type operators. In particular, in the one-dimensional case
these asymptotics were established for $\sigma>5/2$ and $|V(x)|\le C(1+|x|)^{-\rho}$ with some $\rho>4$  
(see also \cite{schlag} for an up-to-date review in this direction, in particular for higher dimensions).

In the second part of the present work we   obtain some  new dispersion estimates  for the Klein--Gordon equation. To this end we introduce
the Bessel potential
\[
\mathcal J_{\alpha} = \mathcal F^{-1} (1+|\cdot|^2)^{\alpha/2} \mathcal F,
\]
where $\mathcal F$ is the Fourier transform. Then the generalized Sobolev space
$H^{\alpha,1}_\sigma(\R)$ (cf.\ \cite[Definition 6.2.2]{BL})
is the space of all tempered distributions $f\in\mathcal{S}'(\R)$ for which the norm
\begin{equation}\label{H-def} 
\Vert f\Vert_{H^{\alpha,1}_\sigma}=\Vert \mathcal J_\alpha f\Vert_{L^1_\sigma},
\qquad \alpha,\sigma\in\R,
\end{equation}
is finite. As before $H^{\alpha,1}=H^{\alpha,1}_0$.

\begin{theorem}\label{Main1}
i) Let $V\in L^1_1(\R)$. Then the following decay holds
\begin{equation}\label{full1}
\Vert [\E^{-\I t\mathbf{H}}{\bf P}_c]^{12}\Vert_{H^{\frac 12,1}\to L^{\infty}}
=\mathcal{O}(t^{-1/2}),\quad t\to\infty.
\end{equation}
ii)  Let  $V\in L^1_2(\R)$. Then, in the non-resonant case, the following decay holds
\begin{equation}\label{full1-new}
\Vert [\E^{-\I t\mathbf{H}}{\bf P}_c]^{12}\Vert_{H^{\frac 12,1}_{1}\to L^{\infty}_{-1}}
=\mathcal{O}(t^{-3/2}),\quad t\to\infty.
\end{equation}
Here ${\bf P}_c$ is the orthogonal projection in
$L^2(\R)\oplus L^2(\R)$ onto the continuous spectrum of $\mathbf{H}^2$ and
$[\cdot]^{ij}$ denotes the $ij$ entry of the corresponding matrix operator.
\end{theorem}

We remark that the corresponding decay for the other entries of the matrix operator $\E^{-\I t\mathbf{H}}{\bf P}_c$
can be obtained similarly.

We note that Theorem \ref{Main1} (i) is frequently stated in terms of the Besov space $B^{\frac 12}_{1,1}(\R)$
defined in \cite [Definitions 6.2.2]{BL}. Namely, recalling
$B^{\frac 12}_{1,1}\subset H^{\frac 12,1}$ (see \cite[Theorem 6.2.4]{BL})
shows that \eqref{full1} holds with $B^{\frac 12}_{1,1}$ in place of $H^{\frac 12,1}$.
Similarly, \eqref{full1-new} holds with $B^{\frac 12}_{1,1,1}$ in place of $H^{\frac 12,1}$,
where $B^{\frac 12}_{1,1,1}$  is the corresponding weighted Besov space (to define $B^{\frac 12}_{1,1,1}$
one need replace $L^1$ by $L^1_1$ in the definition of $B^{\frac 12}_{1,1}$). 
As before this follows from $B^{\frac 12}_{1,1,1}\subset H^{\frac 12,1}_1$
(see  for instance \cite[Proposition 3.12]{MV}).

Moreover, as a consequence of \eqref{full1}  we obtain
\begin{equation}\label{n2}
\Vert [\E^{-\I t\mathbf{H}}{\bf P}_c]^{12}\Vert_{B^{\frac 12-\frac 3p}_{p',p'}\to L^p}
=\mathcal{O}(t^{-\frac 12+\frac 1p}),\quad t\to\infty,\quad \frac 1{p'}+\frac 1p=1,
\end{equation}
for any $p\in [2,\infty]$ under the same assumption $V\in L^1_1(\R)$ (see Corollary \ref{interp}).

In three space dimensions $W^{k,p}\to L^{q}$ estimates for the perturbed Klein--Gordon equation were
established by Soffer and Weinstein \cite{SW} (see also \cite{Y} for general space dimensions $n\ge 3$).
In the one-dimensional case $W^{k,p}\to W^{k,q}$ estimates were obtained by Weder \cite{wed2} for $V\in L^1_\gamma$, where $\gamma>3/2$ in the non-resonant case
and $\gamma>5/2$ in the resonant case.
The dispersive estimate of type \eqref{n2} (with $B^{\frac 12-\frac 3p}_{p',p}(V)$ instead of $B^{\frac 12-\frac 3p}_{p',p'}$)
is shown in {\cite{DF}, but again requiring $V\in L_2^1$ in the resonant case.

For the one-dimensional Klein--Gordon equation the  decay $~t^{-3/2}$ in the weighted energy norms
$H^1_\sigma\oplus L^2_\sigma\to H^1_{-\sigma}\oplus L^2_{-\sigma}$ with $\sigma>5/2$
has been obtained by Komech and Kopylova \cite{KK} (see also the survey \cite{K10}).

Note that dispersion estimates of type \eqref{full}--\eqref{n2} play an important role in proving asymptotic
stability of solitons in the associated one-dimensional nonlinear equations \cite{BS}, \cite{KK11}.
For the discrete Schr\"odinger and wave equations we refer to \cite{EKT}.

\section{Continuity properties of the scattering matrix}
\label{wa-sec}
We first introduce the Banach algebra $\mathcal{A}$ of Fourier transforms of integrable functions
\[
\mathcal{A} = \left\{f(k):\,
f(k) = \int_\R \E^{\I k p}\hat{f}(p)dp, \,\hat{f}(\cdot)\in L^1(\R) \right\}
\]
with the norm $\|f\|_{\mathcal{A}}= \|\hat{f}\|_{L^1}$, plus the corresponding unital Banach algebra $\mathcal{A}_1$
\[
\mathcal{A}_1 = \left\{f(k):\,
f(k) =c+ \int_\R \E^{\I k p}\hat{g}(p)dp, \,\hat{g}(\cdot)\in L^1(\R),\,c\in\C\right\}
\]
with the norm $\|f\|_{\mathcal{A}_1}= |c|+\|\hat{g}\|_{L^1}$.
Evidently,  $\mathcal{A}$ is a subalgebra of $\mathcal{A}_1$.
The algebra $\mathcal{A}_1$ can be treated as an algebra of Fourier transforms of functions
$c\delta(\cdot)+\hat g(\cdot)$, where $\delta$ is the Dirac delta distribution and $\hat g\in L^1(\R)$.
Note that if $f\in\mathcal{A}_1\setminus\mathcal{A}$ and $f(k)\not =0$
for all $k\in\R$ then $f^{-1}(k)\in\mathcal{A}_1$ by the Wiener theorem \cite{Wiener}.

Next we recall a few facts from scattering theory \cite{DT}, \cite{Mar} of the Schr\"odinger operator $H$, defined by formula \eqref{Schr}.
Under the assumption
$V\in L^1_1$ there exist Jost solutions $f_\pm(x,k)$ of
\[
H\psi=k^2\psi,\quad  k\in \overline{\C_+},
\]
normalized according to
\[
f_\pm(x,k)\sim  \E^{\pm \I kx},\quad x\to \pm \infty.
\]
These solutions are given by
\begin{equation}\label{Jost}
f_\pm(x,k) = \E^{\pm \I kx} h_\pm(x,k), \qquad h_\pm(x,k)
= 1 \pm \int_{0}^{\pm\infty} B_\pm(x,y) \E^{\pm 2\I k y}dy,
\end{equation}
where $B_\pm(x,y)$ are real-valued and satisfy (see \cite[\S 2 ]{DT} or \cite[\S 3.1]{Mar})
\begin{align}\label{est1}
|B_\pm (x,y)|&\le  \E^{\gamma_\pm(x)}\eta_\pm(x+y),\\
\label{est11}
|\frac{\partial}{\partial x}B_\pm (x,y)\pm V(x+y)|
&\le 2 \E^{\gamma_\pm(x)}\eta_\pm(x+y)\eta_\pm(x),
\end{align}
with
\begin{equation}\label{est2}
\gamma_\pm(x)=\int_{x}^{\pm\infty}(y-x)|V(y)|dy,\quad
\eta_\pm(x)=\pm\int_{x}^{\pm\infty}|V(y)|dy.
\end{equation}
Since $\eta_\pm(x+\cdot)\in L^1(\R)$ we clearly have
\begin{equation} \label{alg1}
h_\pm(x,\cdot) - 1,~~ h'_\pm(x,\cdot)\in\mathcal A,
\quad \forall x\in\R.
\end{equation}
Let
\[
W(\varphi(x,k),\psi(x,k))=\varphi(x,k)\psi'(x,k)-\varphi'(x,k)\psi(x,k)
\]
be the usual Wronskian, and set
\[
W(k)=W(f_-(x,k),f_+(x,k)),\qquad W_\pm(k)=W(f_{\mp}(x,k),f_{\pm}(x,- k)).
\]
The Jost solutions $f_\pm(x,k)$ and their derivatives do not belong to
$\mathcal{A}_1$, as well as their Wronskian $W(k)$.
However, the entries of the scattering matrix, that is, the transmission and reflection coefficients
\[
T(k)= \frac{2\I k}{W(k)},\quad R_\pm(k)= \mp\frac{W_\pm(k)}{W(k)},
\]
turn out to be the elements of this algebra.
\begin{theorem}\label{MT}
If $V\in L^1_1$, then $T(k)-1\in\mathcal{A}$ and $R_\pm(k)\in\mathcal{A}$.
\end {theorem}
\begin{proof}
Since $|T(k)|\le 1$ for $k\in\R$ the Wronskian $W(k)$ can vanish only at the edge of continuous spectrum $k=0$,
which is known as the resonant case. Moreover, the zero is at most of the first order. \\
{\it Step i)}
We first consider the non-resonant case $W(0)\not =0$.
Abbreviate $h_\pm(k):=h_\pm(0,k)$,  $h'_\pm(k):=h'_\pm(0,k)$.
Then \eqref{Jost} implies
\begin{equation}\label{Wr}
W(k)=2\I kh_+(k)h_-(k)+\tilde W(k),\quad \tilde W(k):=h_-(k)h'_+(k)-h'_-(k)h_+(k),
\end{equation}
\begin{equation}\label{Wrpm}
W_\pm(k)=h_{\mp}(k)h'_\pm(-k)-h_\pm(-k)h'_{\mp}(k).
\end{equation}
Moreover, $\tilde W(k),\;W_\pm(k)\in\mathcal{A}$.
Put
\begin{equation}\label{defnu}
\nu(k):=\frac 1{\I k-1}=\int_0^\infty \E^{\I k y}\E^{-y}dy
\end{equation}
and observe that $\nu(k)\in \mathcal{A}$, $k\nu(k)\in\mathcal{A}_1$ and, therefore,
$\nu(k)W(k)\in\mathcal{A}_1$. Since  $\nu(k)W(k)\to 2$ as $k\to\infty$
then $\nu(k)W(k)\in\mathcal{A}_1\setminus\mathcal{A}$.
Moreover, $\nu(k)W(k)\not =0$ for all $k\in\R$, whence
$(\nu(k)W(k))^{-1}\in\mathcal{A}_1$. Furthermore,
$\nu(k)W_\pm(k)\in\mathcal{A}$ and we obtain
\[
R_\pm(k)=\mp\frac{\nu(k)W_\pm(k)}{\nu(k)W(k)}\in\mathcal{A},\quad
T(k)=\frac{2\I k\nu(k)}{\nu(k)W(k)}\in\mathcal{A}_1.
\]
Moreover, since $T(k)\to 1$ as $k\to\infty$ then $T(k)-1\in\mathcal{A}$.

{\it Step ii)}
In the resonant case we need to work a bit harder. Introduce the functions
\begin{equation}\label{Phi}
\Phi_\pm(k):=h_\pm(k)h'_\pm(0)-h'_\pm(k)h_\pm(0),
\end{equation}
\begin{equation}\label{defK}
K_\pm(x):=\pm\int_x^{\pm\infty}B_\pm(0,y)dy,\quad
D_\pm(x):=\pm\int_x^{\pm\infty}\frac{\partial}{\partial x}B_\pm(0,y)dy,
\end{equation}
where $B_\pm(x,y)$ are the transformation operators from \eqref{Jost}.
Integrating \eqref{Jost} formally by parts we obtain
\begin{align*}
h'_\pm(k)&=\pm\int_0^{\pm\infty}\frac{\partial}{\partial x}B_\pm(0,y)\E^{\pm 2\I ky}dy
=D_\pm(0)+2\I k\int_0^{\pm\infty}D_\pm(y)\E^{\pm 2\I ky}dy\\
&=h'_\pm(0)+2\I k\int_0^{\pm\infty}D_\pm(y)\E^{\pm 2\I ky}dy,\\
h_\pm(k)&=h_\pm(0)+2\I k\int_0^{\pm\infty}K_\pm(y)\E^{\pm 2\I ky}dy.
\end{align*}
We emphasize that the above integrals have to be understood as improper integrals.
Inserting them into \eqref{Phi} gives
\[
\Phi_\pm(k)=2\I k\Psi_\pm(k),\quad
\Psi_\pm(k):=\int_0^{\pm\infty}(D_\pm(y)h_\pm(0)-K_\pm(y)h'_\pm(0))\E^{\pm 2\I ky}dy.
\]
\begin{lemma}\label{lem:gus}
If $V\in L^1_1$ then $\Psi_\pm(k)\in \mathcal{A}$.
\end{lemma}
\begin{proof}
Following \cite{Gus} we will prove that the functions
\[
H_\pm(y):=D_\pm(y)h_\pm(0)-K_\pm(y)h'_\pm(0)
\]
 satisfy $H_\pm\in L^1(\R_\pm)\cap L^{\infty}(\R_\pm)$.
We simplify the original proof of \cite{Gus} using the Gelfand--Levitan--Marchenko
equation in the form proposed in \cite{DT}. Namely, as is known (\S 3.5 in \cite{Mar}) the kernels $B_{\pm}(x,y)$ solve the equations
\begin{equation}\label{GLM}
F_\pm(x+y)+B_\pm(x,y)\pm\int_0^{\pm\infty}B_\pm(x,t)F_\pm(x+y+z)dz=0,
\end{equation}
where the functions $F_\pm(x)$ are absolutely continuous with $F'_\pm\in L^1(\R_\pm)$ and
\begin{equation}\label{FF}
|F_\pm (x)|\le C\eta_\pm(x),\quad \pm x\ge 0,
\end{equation}
with $\eta_\pm$ from \eqref{est2}. Now differentiate \eqref{GLM} with respect to $x$
and set $x=0$. Also set $x=0$ in \eqref{GLM} and then integrate both equations with respect to $y$ from
$x$ to $\pm\infty$. Then \eqref{defK} implies
\[
\pm\int_x^{\pm\infty}F_\pm(y)dy+K_\pm(x)+\int_0^{\pm\infty}B_\pm(0,z)
\int_x^{\pm\infty}F_\pm(y+z)dy\,dz=0
\]
and
\begin{align*}
\mp F_\pm(x)+ D_\pm(x)&+\int_0^{\pm\infty}\frac{\partial}{\partial x}B_\pm(0,z)
\int_x^{\pm\infty}F_\pm(y+z)dy\,dz\\
&-\int_0^{\pm\infty}B_\pm(0,z)F_\pm(x+z)dz=0.
\end{align*}
To get rid of double integration here, we apply  \eqref{defK} and the equalities
\[
\frac{\partial}{\partial z}\int_x^{\pm\infty} F_\pm(y+z)dy=- F_\pm(x+z).
\]
The  integration by parts yields
\begin{align}\label{11}
&\pm(1+ K_\pm(0))\int_x^{\pm\infty}F_\pm(y)dy+K_\pm(x)
\mp\int_0^{\pm\infty}K_\pm(z)F_\pm(x+z)dz\\
\nonumber
&=K_\pm(x)\pm  h_\pm(0)\int_x^{\pm\infty}F_\pm(y)dy
\mp\int_0^{\pm\infty}K_\pm(z)F_\pm(x+z)dz=0
\end{align}
and
\begin{align}\label{12}
&\mp F_\pm(x)+ D_\pm(x)\pm h'_\pm(0)\int_x^{\pm\infty}F_\pm(y)dy
\mp\int_0^{\pm\infty}D_\pm(z)F_\pm(x+z)dz\\
\nonumber
&-\int_0^{\pm\infty}B_\pm(0,z)F_\pm(x+z)dz=0.
\end{align}
Multiplying \eqref{11} by $ h'_\pm(0)$ and \eqref{12} by $ h_\pm(0)$ and
subtracting, we get  integral equations
\begin{equation}\label{13}
H_\pm(x)\mp\int_0^{\pm\infty}H_\pm(y)F_\pm(x+y)dy=G_\pm(x),
\end{equation}
where
\[
G_\pm(x)=h_\pm(0)\Big(\int_0^{\pm\infty}B_\pm(0,y)F_\pm(x+y)dy \pm  F_\pm(x)\Big).
\]
The bounds \eqref{est1} and \eqref{FF} imply
\begin{equation}\label{G-est}
|G_\pm(x)|\le C \eta_\pm(x),\quad \pm x\ge 0.
\end{equation}
Furthermore, for sufficiently large $N>0$ represent  \eqref{13} in the form
\begin{equation}\label{14}
H_\pm(x)\mp\int_{\pm N}^{\pm\infty}H_\pm(y)F_\pm(x+y)dy=G_\pm(x,N),
\end{equation}
where
\[
G_\pm(x,N)=G_\pm(x)\pm\int_0^{\pm N}H_\pm(y)F_\pm(x+y)dy.
\]
 Formulas \eqref{defK} and estimates \eqref{est1}--\eqref{est2} give
$H_\pm\in L^{\infty}(\R_\pm)\cap C(\R_\pm)$.
Then $|G_\pm(x,N)|\le C(N)\eta_\pm(x)$ by \eqref{G-est} and monotonicity of $\eta_\pm(x)$.
Applying the method of successive approximations (cf. \cite[Chapter 3, Section 2]{Mar})
to \eqref{14} we obtain  $H_\pm\in L^{1}(\R_\pm)$.
\end{proof}

Now we can continue the proof of Theorem~\ref{MT} in the resonant case.
Since the Jost solutions are linear dependent at $k=0$, i.e. $h_+(x,0)=c\, h_-(x,0)$,
we distinguish two cases: $h_+(0)h_-(0)\not =0$ and $h_+(0)=h_-(0)=0$.
In the first case $h_+(0)h_-(0)\not =0$ we have
\begin{align*}
\tilde W(k)&=\tilde W(k)-\tilde W(0)=\frac{h_+(k)}{h_-(0)}\Phi_-(k)-\frac{h_-(k)}{h_+(0)}\Phi_+(k)\\
&=2\I k\Big(\frac{h_+(k)}{h_-(0)}\Psi_-(k)-\frac{h_-(k)}{h_+(0)}\Psi_+(k)\Big)
\end{align*}
and similarly in the second case $h_+(0)=h_-(0)=0$ (and thus $h'_+(0)h'_-(0)\not =0$) we have $\Phi_\pm(k)=h_\pm(k)h_\pm'(0)=2\I k \Psi_\pm(k)$ and hence
\[
\tilde W(k)= 2\I k\Big(\frac{h'_+(k)}{h'_-(0)}\Psi_-(k)-\frac{h'_-(k)}{h'_+(0)}\Psi_+(k)\Big).
\]
In summary,
\[
\frac{W(k)}{2\I k}= h_-(k) h_+(k) + \begin{cases} \frac{h_+(k)}{h_-(0)}\Psi_-(k)-\frac{h_-(k)}{h_+(0)}\Psi_+(k), & h_+(0)h_-(0)\not =0,\\
\frac{h'_+(k)}{h'_-(0)}\Psi_-(k)-\frac{h'_-(k)}{h'_+(0)}\Psi_+(k), & h_+(0)h_-(0) =0,\end{cases}
\]
where the right-hand side is in $\mathcal{A}_1$ by \eqref{alg1} and Lemma \ref{lem:gus}.
Since $\frac{W(k)}{2\I k} =T(k)^{-1}\ne~0$ we conclude $T(k)-1\in\mathcal{A}$.
Analogously,
\[
\frac{W_\pm(k)}{2\I k}= \begin{cases}
\frac{h_\pm(-k)}{h_{\mp}(0)}\Psi_{\mp}(k)-\frac{h_\mp(k)}{h_\pm(0)}\Psi_\pm(-k) , & h_+(0)h_-(0)\not =0,\\
\frac{h'_\pm(-k)}{h'_\mp(0)}\Psi_\mp(k)-\frac{h'_\mp(k)}{h'_\pm(0)}\Psi_\pm(-k), & h_+(0)h_-(0) =0.\end{cases}
\]
where the right-hand side is again in $\mathcal{A}$ and hence $R_\pm(k) = \mp \frac{W_\pm(k)}{2\I k} T(k) \in\mathcal{A}$.
\end{proof}

Finally, we will investigate the function
\begin{equation}\label{psi}
\psi(x,y,k)= h_+(y,k)h_-(x,k) T(k) -1, \qquad y\ge x,
\end{equation}
and $\psi(x,y,k)=\psi(y,x,k)$ for $y<x$. From Theorem \ref{MT} and formula \eqref{alg1} it follows that $\psi(x,y,\cdot)\in\mathcal{A}$.

\begin{lemma}\label{lemconst}
The following estimate is valid
\begin{equation}\label{thh}
\| \psi(x,y, \cdot)\|_{\mathcal{A}} \le C,
\end{equation}
with some constant $C$, which does not depend on $x$ and $y$.
\end{lemma}

\begin{proof}
Introduce
\[
\sup_{\pm x \geq 0}\left(\pm\int_{0}^{\pm\infty }|B_\pm(x,y)|dy\right)=C_\pm,
\]
which is finite by \eqref{est1}. Then
\begin{equation}\label{est31}
\|h_\pm(x,\cdot)\|_{\mathcal{A}_1}\leq 1+C_\pm,\quad \|h_\pm(x,\cdot) - 1\|_{\mathcal A}\leq C_\pm,\quad\mbox{for}\quad\pm x\geq 0.
\end{equation}
Now consider the three possibilities
(a) $x\leq y\leq 0$,
(b) $0\leq x\leq y$, and
(c) $x\leq 0\leq y$.
In the case (c) the estimate $\| \psi(x,y,\cdot)\|_{\mathcal{A}}\leq C$
follows immediately from \eqref{est31} and Theorem \ref{MT}.
In the other two cases we use the scattering relations
\begin{equation}\label{scat-rel}
T(k)f_\pm(x,k)=R_\mp(k)f_\mp(x,k) + f_\mp(x,-k)
\end{equation}
to get the representation
\begin{equation}\label{psiest}
\psi(k,x,y)=\begin{cases}
 h_-(x,k)\left(R_-(k)h_-(y,k)\E^{-2\I yk} + h_-(y,-k)\right)-1 & x\leq y\leq 0,\\[2mm]
 h_+(y,k)\left(R_+(k)h_+(x,k)\E^{2\I xk} + h_+(x,-k)\right)-1 & 0\leq x\leq y.
\end{cases}
\end{equation}
Observing that for any function $g(k)\in\mathcal A$ and any real $s$ we have $g(k)\E^{\I k s}\in \mathcal A$ with the norm independent of $s$,
establishes \eqref{thh}.
\end{proof}

\section{The Schr\"odinger equation}
\label{ll-sec}
Now we are ready to  prove the dispersive decay estimate \eqref{full} for the Schr\"odinger equation \eqref{Schr}.
For the one-parameter group of \eqref{Schr} the spectral theorem implies
\begin{equation}\label{PP}
   \E^{-\I tH}P_{c}
   =\frac 1{2\pi \I}\int\limits_{0}^{\infty}
   \E^{-\I t\omega}(\cR(\omega+\I 0)- \cR(\omega-\I 0))\,d\omega,
\end{equation}
where $\cR(\omega)=(H-\omega)^{-1}$ is the resolvent of the Schr\"odinger operator $H$
and the limit is understood in the strong sense \cite{tschroe}.
Given the Jost solutions we can express the kernel
of the resolvent $R(\omega)$ for $\omega=k^2\pm \I 0$, $k>0$, as \cite{DT,tschroe}
\[
[\cR(k^2\pm \I 0)](x,y) = - \frac{f_+(y,\pm k) f_-(x,\pm k)}{W(\pm k)}
=\mp\frac{f_+(y,\pm k) f_-(x,\pm k)T(\pm k)}{2\I k}
\]
for all  $x\leq y$ (and the positions of $x,y$ reversed if $x>y$).

Therefore, in the case  $x\le y$, the integral kernel of $\E^{-\I tH}P_{k_0}(H)$ is given by
\begin{align*}
[\E^{-\I tH}P_{k_0}](x,y)&= \frac{\I}{\pi} \int_{-k_0}^{k_0}
\E^{-\I t k^2}\,\frac{f_+(y,k) f_-(x,k)T(k)}{2\I k} k \,dk\\
&=\frac{1}{2\pi} \int_{-k_0}^{k_0} \E^{-\I(t k^2-|y-x|k)}h_+(y,k) h_-(x,k) T(k)dk,
\end{align*}
where $P_{k_0}=P_H([0,k_0^2])$ is the projection onto energies in the interval $[0,k_0^2]$.
Taking the limit $k_0\to\infty$ we obtain
\begin{align}
\nonumber [\E^{-\I tH}P_{c}](x,y) &= \lim_{k_0\to\infty} [\E^{-\I tH}P_{k_0}](x,y)\\
&=\frac{1}{2\pi} \int_{-\infty}^{\infty} \E^{-\I(t k^2-|y-x|k)}h_+(y,k) h_-(x,k) T(k)dk, \label{integr}
\end{align}
where the integral is to be understood as an improper integral for $t\in\R$ (if $\im(t)<0$ the integral converges absolutely and the
limit is of course not needed). In fact, the convergence of the integral for $t\in\R$  will follow from Lemma~\ref{GT} below, and Lemma \ref{lemconst}
will imply  $|[\E^{-\I tH}P_{k_0}](x,y)| \le C |t|^{-1/2}$. Thus,  we can use dominated convergence to
conclude that the right-hand side of \eqref{integr} is indeed the kernel of $\E^{-\I tH}P_{c}$ on
$L^1(\R)\cap L^2(\R)$.

\begin{lemma}\label{GT}
Let $\psi(x,y,k)$ be defined by  \eqref{psi} and let $\hat\psi(x,y,p)$ be its Fourier transform  with respect to $k$. Then the following representation is valid for $\im(t)\le 0$:
\begin{equation}\label{GT-rep}
[\E^{-\I tH}P_{c}](x,y) = \frac{1}{\sqrt{4\pi\I t}}\left( \E^{-\frac{|x-y|^2}{4\I t}} +
\int_\R \E^{-\frac{(p+|x-y|)^2}{4\I t}} \hat{\psi}(x,y,p) dp\right).
\end{equation}
\end{lemma}
\begin{proof}
By \eqref{integr} we have
\begin{align*}
[\E^{-\I t H}P_{c}](x,y)& =
 \frac{1}{2\pi} \int_{-\infty}^{\infty} \E^{-\I (t k^2 -|y-x| k)} (1+\psi(x,y,k)) dk.
\end{align*}
Since the
first part of the integral is easy to compute we only focus on the second part containing $\psi$. Using
Fubini's theorem this integral is given by
\begin{align*}
 &\frac{1}{2\pi}\lim\limits_{k_0\to\infty}  \int_{-k_0}^{k_0}  \int_\R \E^{-\I (t k^2 -|y-x| k - kp)} \hat{\psi}(x,y,p) dp\, dk =\\
&\frac{1}{2\pi} \lim\limits_{k_0\to\infty}  \int_\R \E^{\I\frac{(p+|y-x|)^2}{4t}}
\int_{-k_0}^{k_0} \E^{-\I \frac{(2kt-|y-x|-p)^2}{4t}} dk\, \hat{\psi}(x,y,p) dp =\\
& \quad = \frac{1}{2\sqrt{4\pi\I t}} \lim\limits_{k_0\to\infty} \int_\R \E^{\I\frac{(p+|y-x|)^2}{4t}}
\big( \erf(q_+) + \erf(q_-)\big) \hat{\psi}(x,y,p) dp,
\end{align*}
where $q_\pm = \frac{k_0}{2} \sqrt{4\I t} \pm \I\frac{(p+|x-y|)}{\sqrt{4\I t}}$
and $\erf(z)$ is the error function \cite[\S 7.2]{dlmf}.
Using $\erf(z) = 1 + O(\E^{-z^2})$ as $z\to\infty$ with $|\arg(z)|< \frac{3\pi}{4}$ (\cite[(7.12.1)]{dlmf})
the claim follows from dominated convergence.
\end{proof}

\begin{proof}[Proof of Theorem~\ref{Main}]
Since 
\[
\Vert\E^{-\I tH}P_{c}\Vert_{L^1\to L^{\infty}}=\sup\limits_{\Vert f\Vert_{L^1}=1,\Vert g\Vert_{L^1}=1}
\langle f, \E^{-\I tH}P_{c}g\rangle
= \sup\limits_{x,y}|[\E^{-\I tH}P_{c}](x,y)|
\]
the claim follows from Lemmas \ref{GT} and \ref{lemconst}.
\end{proof}

In fact we have established the slightly stronger result which covers also the heat semigroup:

\begin{corollary}
Let $V\in L^1_1(\R)$. Then \[
\Vert \E^{- \I tH}P_{k_0}\Vert_{L^1\to L^\infty} \le C |t|^{-1/2},\quad \im(t)\le 0,
\]
for every $0\le k_0\le \infty$.
\end{corollary}

\begin{proof}
Using the representation for $[\E^{-\I tH}P_{k_0}](x,y)$ from the proof of Lemma~\ref{GT}, together with
boundedness of $\erf(q_\pm)$ in the region under consideration and \eqref{thh}, shows $|[\E^{-\I tH}P_{k_0}](x,y)| \le C |t|^{-1/2}$
as desired.
\end{proof}

\section{The Schr\"odinger equation (non-resonant case)}
\label{lln-sec}
In this section we consider the non-resonant case and prove the dispersive decay estimate \eqref{full-new}.
We begin by representing the jump of the resolvent across the spectrum as
\[
[\cR(k^2+\I 0)- \cR(k^2-\I 0)](x,y)=
\frac{ T(k)f_+(y,k) f_-(x,k) + \overline{T(k)f_+ (y,k)} \overline{f_- (x,k)}}{-2\I k},
\]
for $x\le y$ and $k>0$.
The scattering relations \eqref{scat-rel} imply
\begin{align*}
 f_-(x,k) &=T(-k) f_+(x,-k)-R_-(-k)f_-(x,-k),\\
\overline{f_+(y,k)} &=T(k) f_-(y,k)-R_+(k)f_+(y,k)
\end{align*}
and using the consistency relation $T \overline R_- + \overline {T } R_+=0$
we arrive at the formula (cf.\ \cite[p.13]{schlag})
\begin{equation}\label{WS}
[\cR(k^2+\I 0)- \cR(k^2-\I 0)](x,y)=\frac{|T(k)|^2}{-2\I k}
[f_+(y,k)f_+(x,-k)+f_-(y,k)f_-(x,-k)].
\end{equation}
Inserting this into \eqref{PP} gives
\begin{align*}
[\E^{-\I tH}P_{c}](x,y) &= [\cK_+(t)](x,y)+[\cK_-(t)](x,y),\\ 
[\cK_\pm(t)](x,y) &= \frac{1}{4\pi} \int_{-\infty}^{\infty} \E^{-\I (tk^2 \mp |y-x| k})|T(k)|^2 h_\pm(y,k) h_\pm(x,-k) dk.
\end{align*}
Using integration by parts we get
\begin{align*}
 [\cK_\pm(t)](x,y)
=& \pm\frac{|y-x|}{8\pi t} \int_{-\infty}^{\infty} \E^{-\I( tk^2\mp|y-x|k)}\frac{|T(k)|^2}{k}h_\pm(y,k) h_\pm(x,-k) dk\\
& {} -\frac{1}{8\pi\I t} \int_{-\infty}^{\infty} \E^{-\I(t k^2\mp|y-x|k)}
\frac{|T(k)|^2}{k^2}h_\pm(y,k) h_\pm(x,-k) dk\\
& {} +\frac{1}{8\pi\I t} \int_{-\infty}^{\infty} \E^{-\I(t k^2\mp|y-x|k)}\frac
{\frac{\partial}{\partial k}\Big[|T(k)|^2h_\pm(y,k) h_\pm(x,-k)\Big]}{k} dk
\end{align*}
Applying the arguments from the proof of Lemma \ref{GT} we obtain
\begin{equation}\label{cK-rep}
 [\cK_\pm(t)](x,y)=\frac{t^{-3/2}}{8\sqrt{\pi\I}}\int_{\R}
 \E^{\I\frac{(p+|x-y|)^2}{4t}} \sum\limits_{j=1}^3\hat{\psi}_j^{\pm}(x,y,p) dp
\end{equation}
where $\hat{\psi}_j^{\pm}(x,y,p)$, $j=1,2,3$, are the Fourier transforms of the functions
\begin{align*}
\psi_1^{\pm}(x,y,k)&=\pm|y-x|\frac{|T(k)|^2}{k}h_\pm(y,k) h_\pm(x,-k),\\
\psi_2^{\pm}(x,y,k)&=\I\frac{|T(k)|^2}{k^2}h_\pm(y,k) h_\pm(x,-k),\\
\psi_3^{\pm}(x,y,k)&=-\I\frac{\frac{\partial}{\partial k}\Big[|T(k)|^2h_\pm(y,k) h_\pm(x,-k)\Big]}{k},
\end{align*}
respectively. To estimate their $\mathcal A$ norms we first show
\begin{lemma}\label{lemma-n}
Let $V\in L^1_2$ and  $W(0)\neq 0$. Then
$T(k)h _\pm(x,k)/k \in \mathcal A$,  and
\begin{equation}\label{Test1}
\Big\Vert\frac{T(k)h _\pm(x,k)}{k}\Big\Vert_{\mathcal A}\le C (1 + |x|),\quad x\in \R.
\end{equation}
\end{lemma}
\begin{proof}
Since $\frac{T(k)}{k}=\frac{2\I\nu(k)}{\nu(k)W(k)}\in \mathcal A$ (recall \eqref{defnu}), then for $x\in\R_\pm$ the
bound \eqref{Test1} follows from \eqref{est31}.
Consider the case $x\in \R_\mp$. The scattering relations \eqref{scat-rel} imply
\begin{align}\nonumber
T(k)h_\pm(x,k)
= &(R_\mp(k)+1)h_\mp(x,k)\E^{\mp 2\I k x}-
(h_\mp(x,k)-h_\mp(x,-k))\E^{\mp 2\I k x} \\\label{proper8}
&  + h_\mp(x,-k)(1 -\E^{\mp 2\I k x}).
\end{align}
Using \eqref{Jost} we obtain
\begin{align*}
&\frac{h_\mp(x,k)-h_\mp(x,-k)}{k}=
\mp\int\limits_{0}^{\mp\infty}B_\mp(x,r)\frac{\E^{\mp\I kr}-\E^{\pm\I kr}}{k}dr\\
&=\I\int\limits_{0}^{\mp\infty}B_\mp(x,r)\int\limits_{-r}^{r}\E^{\I ky}dy\, dr
= \I\int\limits_{-\infty}^{\infty}\Big(\int\limits_{\mp |y|}^
{\mp\infty} B_\mp(x,r)dr\Big)\E^{\I ky}dy.
\end{align*}
Next, observe that formula  \eqref{est1} implies that
if $V\in L_2^1$, then $B_\mp(x,.)\in L_{1}^1(\R_\mp)$ for any fixed $x$,
and consequently
\[
S_\mp(x,y)=\int_{y}^{\mp\infty}|B_\mp(x,r)|dr\in L^1(\R_\mp).
\]
Based on this observation we get
\begin{equation}\label{proper5}
\Big\Vert\frac{h_\mp(x,k)-h_\mp(x,-k)}{k}\Big\Vert_{\mathcal A}\le C ,\quad x\in\R_\mp.
\end{equation}
By the same reasons formula \eqref{est11} implies
\begin{equation}\label{proper55}
\frac{h_\mp^\prime(0,k)-h_\mp^\prime(0,-k)}{k}\in\mathcal A.
\end{equation}
Next, from \eqref{Wr} and \eqref{Wrpm} it follows
\begin{align*}
\frac{W(k)\mp W_\pm(k)}{k} =&2\I h_+(k)h_-(k) +\\
& {} +\frac{h_-(k) h_+^\prime(k) - h_-^\prime(k)h_+(k)\mp
h_\mp(k)h^\prime_\pm(-k) \pm h_\pm(-k)h_\mp^\prime(k)}{k}=\\
=& 2\I h_+(k)h_-(k)\pm h_\mp(k)\frac{h^\prime_\pm(k) -h^\prime_\pm(-k)}{k}\mp
h_\mp^\prime(k)\frac{h_\pm(k) -h_\pm(-k)}{k}.
\end{align*}
Applying \eqref{alg1}, \eqref{proper5}, \eqref{proper55} we get
\[
\left(\frac{W(k)\mp W_\pm(k)}{k} -2\I\right)\in\mathcal A.
\]
As is shown in Theorem \ref{MT} in
the non-resonant case $W^{-1}(k)\in\mathcal A$. Thus \begin{equation}\label{proper9}
\frac{R_\pm(k)+1}{k}=\frac{1}{W(k)}\frac{W(k)\mp W_\mp(k)}{k}\in\mathcal A.
\end{equation}
Next, $\frac{1 -\E^{\mp 2\I kx}}{\I k}$ is the Fourier transform of
the indicator function of $[0,2x]$, therefore
\begin{equation}\label{proper10}
\Big\Vert\frac{1 -\E^{\mp 2\I k x}}{k}\Big\Vert_{\mathcal A}\leq  2|x|.
\end{equation}
Finally, substituting \eqref{proper5}, \eqref{proper9}, and \eqref{proper10} into \eqref{proper8} we obtain \eqref{Test1}.
\end{proof}
Since we have already seen the estimate $\|T(k) h_\pm(x,k)\|_{\mathcal A}\le C$ in the proof of Theorem~\ref{Main},
this lemma immediately implies that
\begin{equation}\label{eq:12}
\Vert \hat{\psi}_j^{\pm}(x,y,\cdot)\Vert_{L^1}\le C(1+|x|)(1+|y|),\quad j=1,2.
\end{equation}
To estimate $\Vert \hat{\psi}_3^{\pm}(x,y,\cdot)\Vert_{L^1}$
we need one more property.
\begin{lemma}\label{lem:proizv} Let $V\in L_2^1$ and $W(0)\neq 0$.
Then $\frac{\partial}{\partial k}(T(k)h_\pm(x,k))\in\mathcal A$ with
\[
\Big\Vert \frac{\partial}{\partial k}(T(k)h_\pm(x,k))\Big\Vert_{\mathcal A}\leq C(1+|x|),\quad x\in\R.
\]
\end{lemma}
\begin{proof}
The representation \eqref{Jost} and the bounds \eqref{est1}--\eqref{est11} imply
\begin{equation} \label{alg-dif}
\frac{\partial}{\partial k} h_\pm(x,k)=\dot h_\pm(x,k)\in \mathcal A,\quad
\frac{\partial}{\partial k} h'_\pm(x,k)
\in \mathcal A \quad\text{if } V\in L^1_2
\end{equation}
with 
\begin{equation}\label{est11-d} 
\|\frac{\partial}{\partial k} h'_\pm(x,\cdot)\|_{\mathcal A}+\|\dot h_\pm(x,\cdot)\|_{\mathcal A}\leq C,\quad x\in \R_\pm.
\end{equation}
Therefore $\frac{d}{d k} W_{\pm}(k):=\dot W_{\pm}(k)\in \mathcal A$. 
Further, from \eqref{Wr} and \eqref{alg-dif} it follows that
$\nu(k)\dot W(k)\in\mathcal A$, where $\nu(k)$ is defined by \eqref{defnu}. Since in the nonresonant case $(\nu(k)W(k))^{-1}\in\mathcal A_1$ and $W^{-1}(k)\in\mathcal A$ then
\begin{equation}\label{imp12}
\dot T(k)=\frac{1}{W(k)}\left(2\I-\dot W(k)T(k)\right)\in\mathcal A,\quad \dot R_\pm(k)\in\mathcal A.
\end{equation}
Thus for $x\in \mathbb R_\pm$ the statement of the Lemma is evident in view of \eqref{est31}, \eqref{imp12}
and \eqref{est11-d}.
To get it for $x\in \mathbb R_\mp$ we use \eqref{imp12}, \eqref{est11-d}, and again the scattering relations \eqref{scat-rel}
which gives
\[
\frac{\partial}{\partial k}(T(k)h_\pm(x,k))=\E^{\mp 2\I kx}\left(\frac{\partial}{\partial k} (R_\mp(k)h_\mp(x,k)) \mp 2\I x R_\mp(k)h_\mp(x,k)\right)+\dot h_\mp(x,-k).
\]
\end{proof}
As pointed out in the proof of Theorem \ref{Main} the estimate $\|T(k) h_\pm(x,k)\|_{\mathcal A_1}\le C$ is valid for $x\in\R$.  This and  Lemma \ref{lem:proizv} imply
\begin{equation}\label{3}
\Vert \hat{\psi}_3^{\pm}(x,y,\cdot)\Vert_{L^1}\le C(1+|x|)(1+|y|).
\end{equation}
Finally, combining \eqref{cK-rep}, \eqref{eq:12},  \eqref{3} and Lemma \ref{GT} we obtain
\[
|[\cK_\pm(t)](x,y)|\le Ct^{-3/2}(1+|x|)(1+|y|),\quad t\ge 1,
\]
which shows \eqref{full-new} and finishes the proof of Theorem \ref{Main-new}.
\section{The Klein--Gordon equation}
\label{KG-sect}
In this section  we prove the estimate \eqref{full1} for the Klein--Gordon equation \eqref{KGEv}.
We estimate the low-energy and high-energy components of the solution separately. Equation \eqref{full1}
will immediately follow from the two theorems below.
\begin{theorem}\label{thKG1}
Assume $V\in L^1_1(\R)$. Then for any smooth function $\zeta$  with bounded support the following decay holds
\[
\big\|\E^{-\I t\mathbf{H}}{\bf P}_c\,\zeta(\mathbf{H}^2)\big\|_{L^1\to L^\infty}
={\mathcal O}(t^{-1/2}),\quad t\to \infty.
\]
\end{theorem}
\begin{theorem}\label{thKG} Assume $V\in L^1_1(\R)$ and let $\xi(x)$ be a smooth function such that $\xi(x)=0$ for $x\le m^2+1$ and
$\xi(x)=1$ for $x\geq m^2 + 2$. Then
\[
\big\|[\E^{-\I t\mathbf{H}}]^{12}\,\xi(\mathbf{H}^2)\big\|_{H^{\frac 12,1}\to L^\infty}={\mathcal O}(t^{-1/2}),
\quad t\to \infty.
\]
\end{theorem}
As a consequence of  \eqref{full1} we get 
\begin{corollary}\label{interp}
Assume $V\in L^1_1(\R)$. Then \eqref{n2} holds for any $p\in [2,\infty]$. Namely 
\begin{equation}\label{interp-est}
\Vert [\E^{-\I t\mathbf{H}}{\bf P}_c]^{12}\Vert_{B^{\frac 12-\frac 3p}_{p',p'}\to L^p}
=\mathcal{O}(t^{-\frac 12+\frac 1p}),\quad t\to\infty,\quad \frac 1{p'}+\frac 1p=1.
\end{equation}
\end{corollary}
\begin{proof}
Recall that the Klein--Gordon equation preserves the energy
\[
\|\dot{\psi}\|_{L^2}^2 + \spr{\psi}{H\psi}_{L^2} + m^2 \|\psi\|_{L^2}^2.
\]
Since $[\E^{-\I t\mathbf{H}}{\bf P}_c]^{12} \pi_0$ corresponds to the initial condition
$(\psi(0),\dot{\psi}(0))=(0,\pi_0)$ with $\pi_0 = P_c(H)\pi_0$ we obtain the estimate
$\spr{\psi}{H\psi}_{L^2} + m^2 \|\psi\|_{L^2}^2 \le \|\pi_0\|_{L^2}^2$
in this case. Moreover, since for $V\in L^1$ the multiplication operator $V$ is relatively form bounded
with bound $0$ with respect to $H_0=-\frac{d^2}{dx^2}$ (\cite[Lemma~9.33]{tschroe}), the graph norms
of $H$ and $H_0$ are equivalent, and we obtain $\|\psi\|_{H^1} \le C \|\pi_0\|_{L^2}$.
Hence by duality we also get 
\begin{equation}\label{fulle}
\Vert [\E^{-\I t\mathbf{H}}{\bf P}_c]^{12}\Vert_{H^{-1}\to L^2}
=\mathcal{O}(1),\quad t\to\infty,\quad H^{-1}=H^{-1,2}
\end{equation}
Since $H^{-1}=B^{-1}_{2,2}$ due to \cite[Theorem 2.3.2 (d)]{T},
real interpolation between \eqref{full1} and \eqref{fulle} gives 
\eqref{interp-est} .
\end{proof}
\subsection{Low-energy decay}
Here we prove Theorem \ref{thKG1}.
We will need a small variant of the van der Corput lemma which is of independent interest.
\begin{lemma}\label{lem:vC}
Consider the oscillatory integral
\[
I(t) = \int_a^b \E^{\I t \phi(k)} f(k) dk,
\]
where $\phi(k)$ is real-valued function.
If $\phi''(k)\not =0$ in [a,b] and $f\in\mathcal{A}_1$, then
\[
|I(t)| \le C_2 [t\min_{a\le k\le b}|\phi''(k)|]^{-1/2}\|f\|_{\mathcal A_1}, \quad t\ge 1.
\]
where $C_2\le 2^{8/3}$ is the optimal constant from the van der Corput lemma.
\end{lemma}
\begin{proof}
Writing $f(k)=c+\int_\R e^{\I k y}\hat g(y)dy$ we have
\[
I(t) =
\int_{\R} \hat{g}(y) I_{y/t}(t)dy + c I_0(t), \quad I_v(t)= \int_a^b \E^{\I t (\phi(k) + vk)} dk.
\]
By the van der Corput lemma
\[
|I_v(t)| \le C_2 [t\min_{a\le k\le b}|\phi''(k)|]^{-1/2}, \quad t\ge 1,
\]
where $C_2\le 2^{8/3}$ (cf.\ \cite{rog}) and the claim follows from the definition of the norm in $\mathcal A_1$.
\end{proof}
Note that the analogous lemma extends to higher derivatives and to unbounded intervals (where the integral has to
be understood as an improper Riemann integral).

The resolvent $\mathbf{R}(\omega)$ of the operator \eqref{H} associated with the Klein-Gordon equation \eqref{KGE} can be expressed in terms of the resolvent of the Schr\"odinger
operator $\cR(\omega)=(H-\omega)^{-1}$ as
\[
 \mathbf{R}(\omega)=
   \begin{pmatrix}
 0            &  0\\
  -\I          &  0
\end{pmatrix}
+
\begin{pmatrix}
 \omega &  \I \\
  -\I \omega^2&  \omega
\end{pmatrix} \cR(\omega^2-m^2).
\]
For $\E^{-\I t\mathbf{H}}{\bf P}_c\, \zeta(\mathbf{H}^2)$ the spectral representation of type \eqref{PP} 
holds:
\begin{align}\nonumber
&\E^{-\I t\mathbf{H}}{\bf P}_c\, \zeta(\mathbf{H}^2)
   =\frac 1{2\pi \I}\int\limits_\Gamma  \E^{-\I t\omega} \zeta(\om^2) ( \mathbf{R}(\omega+\I 0)- \mathbf{R}(\omega-\I 0))\,d\omega\\
&\quad =\frac 1{2\pi \I}\int_\Gamma \E^{-\I\om t} \zeta(\omega^2)
\begin{pmatrix}
 \omega &  \I \\\label{PP1}
  -\I \omega^2&  \omega
\end{pmatrix}\left(\cR((\omega+\I0)^2-m^2)-\cR((\omega-\I0)^2-m^2)\right)d\om,
\end{align}
where $\Gamma=(-\infty,-m)\cup (m,\infty)$.
Denote
\[
{\mathcal M}_t(k)=\begin{pmatrix}
 \cos(t\sqrt{k^2 + m^2})  & \ds\frac{\sin(t\sqrt{k^2 + m^2})}{\sqrt{k^2+m^2}}\\
  -\sqrt{k^2+m^2}\sin(t\sqrt{k^2 + m^2})&  \cos(t\sqrt{k^2 + m^2})
\end{pmatrix},
\]
then \eqref{PP1} can be rewritten as
\begin{equation}\label{E-rep}
[\E^{-\I t\mathbf{H}} \mathbf{P}_c\,\zeta(\mathbf{H}^2)](x,y)
=\frac{1}{2\pi}\int_{-\infty}^{\infty}{\mathcal M}_t(k)\,
\E^{\I|y-x|k}\zeta(k^2+m^2)\, (\psi(x,y,k)+1) dk,
\end{equation}
where
the function $\psi(x,y,k)$ is defined by \eqref{psi}.
We obtain oscillatory integrals with the  phase functions $\phi_{\pm}(k)=\pm\sqrt{k^2+m^2}-vk$, where
$v=\frac{|y-x|}t$. The second derivative of $\phi_{\pm}(k)$ satisfies
\[
|\phi''_{\pm}(k)|=\frac{m^2}{\sqrt{(k^2+m^2)^3}}\ge C (m,\zeta),
\quad (k^2 + m^2)\in{\rm supp}\,\zeta.
\]
Since $(k^2 +m^2)^{j/2}\zeta(k^2 + m^2)\in\mathcal A$ for $j=-1,0,1$ and   $\| \psi(x,y,k)\|_{\mathcal{A}}\leq C$ due to  \eqref{thh}, then
Lemma~\ref{lem:vC} implies
\[
\max_{x,y\in\R}|[\E^{-\I t\mathbf{H}} \mathbf{P}_c\,\zeta(\mathbf{H}^2)](x,y)|\le Ct^{-1/2},\quad t\ge 1.
\]
\subsection{High-energy decay}
Here we prove Theorem \ref{thKG}.
Our  proof is based on the following version of \cite[Lemma 2]{MSW}:
\begin{lemma}\label{l1}
Let $\eta(k)$, $k\geq 1$, be a smooth function such that
 $|\eta^{(j)}(k)|\le k^{-j}$ for $j=0,1$. Then for any $g(k)\in \mathcal A_1$,
 $\alpha> 3/2$ and $t\ge 1$
\begin{equation}\label{OI-est}
\sup\limits_{p\in\R}\Big|\int_{1}^\infty\eta(k)\frac{\E^{\pm\I t\sqrt{k^2+m^2}+ \I k p}}
{k^\alpha}g(k)dk\Big|\le C \Vert g\Vert_{\mathcal A_1} t^{-1/2}.
\end{equation}
Moreover,
\begin{equation}\label{OI-est11}
\sup\limits_{p\in\R}\Big|\int_{1}^\infty\eta(k)\frac{\E^{\pm\I t\sqrt{k^2+m^2}+\I k p}}
{k^{3/2}}dk\Big|\le C t^{-1/2}.
\end{equation}
Here the constants $C$ depend on the parameters $m$ and $\alpha$ only.
\end{lemma}
\begin{proof}
Consider "+" case and set $v=-p/t$. To prove \eqref{OI-est} we have to estimate the oscillatory integral
\begin{equation*}
I_\alpha(t)=\int_{1}^\infty k^{-\alpha}\eta(k)\E^{\I t\phi(k)}g(k)dk
\end{equation*}
with the phase function $\phi(k)=\sqrt{k^2+m^2}-vk$.
Split the integral according to
\[
I_\alpha(t)=I_\alpha^1(t)+I_\alpha^2(t)=\int_{1}^{t} +\int_{t}^{\infty}.
\]
Since $\|g\|_\infty\leq \|g\|_{\mathcal A_1}$ we obtain
\begin{equation}\label{It2}
|I_\alpha^2(t)|\le  \Vert g\Vert_{\mathcal A_1} \int_{t}^{\infty}k^{-\alpha} dk
\le C\Vert g\Vert_{\mathcal A_1} t^{1-\alpha}.
\end{equation}
To estimate $I_\alpha^1(t)$ we abbreviate
\[
\Psi(k,t)=\int_{1}^k\E^{\I t\phi(\tau)}g(\tau)d\tau.
\]
Since
\[
\min\limits_{1\le \tau\le k}\phi''(\tau)=\phi''(k)=
\frac{m^2}{(\sqrt{k^2+m^2})^3}\ge \frac{C} {k^3}
\]
Lemma~\ref{lem:vC} implies
\begin{equation}\label{estpsi}
|\Psi(k,t)|\le C \Vert g\Vert_{\mathcal A_1} t^{-1/2}k^{3/2}.
\end{equation}
Integrating $I_\alpha^1(t)$ by parts we get
\[
|I_\alpha^1(t)| \le |\Psi(t,t)| t^{-\alpha} + \int_{1}^{t}|\Psi(k,t)| |\Lambda(k)|dk,
\]
where $\Lambda(k)=\frac{k \eta'(k) - \alpha \eta(k)}{k^{\alpha+1}}$ is a smooth bounded function and $\Lambda(k)=O(k^{-\alpha-1})$ as $k\to\infty$. By \eqref{estpsi}
\[
|I_\alpha^1(t)|\le C\Vert g\Vert_{\mathcal A_1} \left(t^{1-\alpha} + (1+\alpha) t^{-1/2} \int_{1}^{t}k^{1/2 - \alpha}dk\right)
\le C \Vert g\Vert_{\mathcal A_1}t^{-1/2}.
\]
Together with \eqref{It2} this proves \eqref{OI-est}.

Next we turn to \eqref{OI-est11}. Since \eqref{It2} is valid for $\alpha=3/2$ and $g(k)=1$ this follows
from Lemma~\ref{prop:psi0}.
\end{proof}

To prove Theorem \ref{thKG}
we have to show that for any function smooth function $f\in C_0^\infty$ with compact support
\begin{equation}\label{hKG}
\big\|[\E^{-\I t\mathbf{H}}]^{12}\,\xi(\mathbf{H^2})f\big\|_{L^\infty}\le C
t^{-1/2}\|f\|_{H^{\frac 12,1}},\quad t\ge 1.
\end{equation}
The kernel of the resolvent of the free Schr\"odinger operator reads (\cite[\S 7.4]{tschroe})
\[
[\cR_0(k^2\pm \I0)](x,y)=\pm\I\E^{\pm\I k |x-y|}/(2 k),\quad k >0.
\]
Substituting the second resolvent identity $\cR(\lambda)= \cR_0(\lambda)- \cR_0(\lambda)V \cR(\lambda)$
into the 12 entry of \eqref{E-rep}, and taking into account that $\xi(x)=0$ for $x\leq m^2+1$, we obtain
\[
[\E^{-\I t\mathbf{H}}]^{12}\,\xi(\mathbf{H}^2)=\mathbf{K}_0(t)+\mathbf{K}_1(t),
\]
where the kernels of the operators $\mathbf{K}_0(t)$ and $\mathbf{K}_1(t)$ read
\begin{align} \label{bK1}
[\mathbf{K}_0(t)](x,y)=&\frac 1{2\pi}\int_{|k|\geq 1}\xi(k^2+m^2)
\frac{\sin(t\sqrt{k^2 + m^2})}{\sqrt{k^2+m^2}}\E^{\I k(x-y)}dk,\\\nonumber
 [\mathbf{K}_1(t)](x,y)=&\frac {\I}{4\pi }\int\limits_{\R} V(z)\left(\int_{|k|\geq 1}\xi(k^2+m^2)\right.\\
\label{bKn}
&\left.\times\frac{\sin(t\sqrt{k^2 + m^2})}{\sqrt{k^2+m^2}}
\frac{\E^{\I k(|x-z|+|z-y|)}}{k}(\psi(y,z,k)+1)dk\right) dz.
\end{align}
Note that the derivative $\xi^\prime(x)$ has support inside the set $[m^2+1,m^2+2]$.
Therefore the function
\begin{equation}\label{eta}
\eta(k):=\frac{\I}{4\pi}\xi(k^2+m^2)\frac{k}{\sqrt{k^2+m^2}}
\end{equation}
satisfies the conditions of Lemma \ref{l1}. Applying this lemma with
$\alpha=2$, $g(k)=\psi(y,z,k)+1$, $p= |x-z|+|z-y|$, and taking into account \eqref{thh}, we get
\begin{equation}\label{born11}
\| \mathbf{K}_1(t)f\|_{L^\infty}\le C t^{-1/2}\|f\|_{L^1}\le C|t|^{-1/2}\|f\|_{H^{\frac 12,1}},\quad t\ge 1,
\end{equation}
since $H^{\frac 12,1}\subset L^1$ due to \cite [Theorems 6.2.3]{BL}.
It remains to get  an estimate of type \eqref{hKG} for $\mathbf{K}_0(t)$.
\begin{lemma}\label{Bessel}
Assume $V\in L^1_1$. Then 
\[
\| \mathbf{K}_0(t)f\|_{L^\infty}\le C|t|^{-1/2}\|f\|_{H^{\frac 12,1}},\quad |t|\ge 1.
\]
\end{lemma}
\begin{proof}
For any $f\in C_0^\infty$  \eqref{bK1} implies
\[
\int_{\R}[{\bf K}_0(t)](x,y)f(y)dy=\sum_{\mp}\int_{|k|\geq 1}\frac{\eta(k)}{k(1+k^2)^{1/4}}
\E^{\pm t\sqrt{k^2 + m^2}+\I k x}(1+k^2)^{1/4}\hat f(k)dk,
\] 
where $\eta(k)$ is defined by \eqref{eta}. 
Denote $g=\mathcal J_{\frac 12} f$. By  definition \eqref{H-def} we have 
\begin{equation}\label{bbb}
\|g\|_{L^1}=\|f\|_{H^{\frac 12,1}}.
\end{equation}  
Thus 
\[
\|{\bf K}_0(t)f\|_{L^\infty}\leq C \int_{\R} |g(y)|
 \sup\limits_{x,y\in\R}\Big|\int_{1}^\infty\eta(k)\frac{\E^{\pm\I t\sqrt{k^2+m^2}+\I k (x-y)}}
{k^{3/2}}dk\Big|dy,
\]
which together with \eqref{bbb} and \eqref{OI-est11} implies the required estimate for ${\bf K}_0$.
\end{proof}
Together with \eqref{born11} this establishes Theorem \ref{thKG}.
This finishes the proof of Theorem \ref{Main1} (i).

\section{The Klein--Gordon equation (non-resonant case)}
\label{KGn-sect}
Suppose that the operator $H$ in \eqref{H} has no resonance at $0$.
To prove Theorem \ref{Main1} (ii)  consider first the low-energy part of the solution.
Using the representation \eqref{WS} we rewrite \eqref{E-rep} as follows
\begin{align*}
&[\E^{-\I t\mathbf{H}} \mathbf{P}_c\,\zeta(\mathbf{H}^2)](x,y)\\
&\qquad = \sum_{\sigma_1,\sigma_2\in\{\pm\}} \frac{1}{2\pi}\int_{-\infty}^{\infty}A_{\sigma_1}(k)\E^{\I t\sqrt{k^2+m^2}}\,
\E^{\I|y-x|k}\zeta(k^2+m^2)\, {\mathcal T}_{\sigma_2}(x,y,k) dk
\end{align*}
where
\[
A_{\pm}(k)=\begin{pmatrix}
1  & \mp\ds\frac{\I}{\sqrt{k^2+m^2}}\\
  \pm \I\sqrt{k^2+m^2}& 1
\end{pmatrix},
\]
and ${\mathcal T}_{\pm}(x,y,k)=|T(k)|^2f_{\pm}(k)f_{\pm}(-k)$.
Applying integration by parts we obtain for the summand $[\E^{-\I t\mathbf{H}}\mathbf{P}_c\,\zeta(\mathbf{H}^2)]_{++}(x,y)$
with $A_+$ and ${\mathcal T}_+$:
\begin{align*}
&[\E^{-\I t\mathbf{H}}\mathbf{P}_c\,\zeta(\mathbf{H}^2)]_{++}(x,y)\\
&=-\frac{1}{2\pi\I t}\int_{-\infty}^{\infty}\E^{\I t\sqrt{k^2+m^2}}\frac{\partial}{\partial k}
\Big[\E^{\I|y-x|k}\zeta(k^2+m^2)\frac{\sqrt{k^2+m^2}}{k}A_+(k)\,{\mathcal T}_+(x,y,k)\Big]dk.
\end{align*}
Using the same arguments as in the proof of Theorem \ref{Main-new} (see Section \ref{lln-sec})
we obtain
\[
|[\E^{-\I t\mathbf{H}}\mathbf{P}_c\,\zeta(\mathbf{H}^2)]_{++}(x,y)|\le Ct^{-3/2}(1+|x|)(1+|y|),\quad t\ge 1
\]
and hence
\[
\big\|[\E^{-\I t\mathbf{H}}\mathbf{P}_c]_{++}\zeta(\mathbf{H}^2)\big\|_{L^1_1\to L^\infty_{-1}}
={\mathcal O}(|t|^{-3/2}),\quad t\to \infty.
\]
The other summands  can been estimated  similarly, and we get
\begin{equation}\label{KGn-e}
\big\|[\E^{-\I t\mathbf{H}}\mathbf{P}_c]\zeta(\mathbf{H}^2)\big\|_{L^1_1\to L^\infty_{-1}}
={\mathcal O}(t^{-3/2}),\quad t\to \infty.
\end{equation}
It remains to consider the high-energy part.
 To simplify notations we denote $c(k,t):=\cos(t\sqrt{k^2 + m^2})$ and $\chi(k):=\xi(k^2 + m^2)$.
Applying integration by parts to  \eqref{bK1} and \eqref{bKn}  we get
\begin{align}\nonumber
[\mathbf{K}_0(t)](x,y) &=\frac 1{2\pi t}\int\limits_{|k|\geq 1} \!\! c(k,t) 
\frac{\partial}{\partial k}\frac{\chi(k)\E^{\I k(x-y)}}{k} dk,\\ \label{pred5}
 [\mathbf{K}_1(t)](x,y) &=\frac {\I}{4\pi t}\int\limits_{\R} V(z) \!\!
\int\limits_{|k|\geq 1} \!\! c(k,t) \frac{\partial}{\partial k}\frac{ \chi(k)
\E^{\I k(|x-z|+|z-y|)}(1+\psi(y,z,k))}{k^2}dk\, dz.
\end{align}
To estimate \eqref{pred5}, recall that 
$\Vert \frac{\partial}{\partial k}\psi(z,y,k)\Vert_{\mathcal A}\le C(1+|z|)(1+|y|)$
by \eqref{eq:12} and \eqref{3}. Moreover,
$|x-z|+|z-y|\le (1+|x|)(1+|y|)(1+2|z|)$. 
Thus  the $\mathcal A_1$ norm of the derivative with respect to $k$ in the integrand of \eqref{pred5} can be estimated by $C(1+|x|)(1+|y|)(1+|z|)$. Moreover, $\chi^\prime(k)$ is a smooth function with a finite support.
Applying Lemma \ref{l1} to the integral with respect to $k$ in \eqref{pred5} and taking into account that $|V(z)|\in L^1_1(\R)$ we come to the estimate
\begin{equation}\label{tmK1}
 [\mathbf{K}_1(t)](x,y)|\le Ct^{-3/2} (1+|x|)(1+|y|),\quad t\ge 1.
\end{equation}
Further,
\begin{align*}
[\mathbf{K}_0(t)](x,y)&=\frac 1{2\pi t}\int_{|k|\geq 1} c(k,t) 
\chi^\prime(k)k^{-1}\E^{\I k(x-y)}dk\\
&-\frac{1}{2\pi t}\int_{|k|\geq 1} c(k,t)\chi(k)\E^{\I k(x-y)}k^{-2}dk\\
&+\frac{\I}{2\pi t}\int_{|k|\geq 1} (x-y)c(k,t)\chi(k)\E^{\I k(x-y)} k^{-1}dk\\
&=[\mathbf{K}_{01}(t)](x,y)+[\mathbf{K}_{02}(t)](x,y)+[\mathbf{K}_{03}(t)](x,y).
\end{align*}

Lemma \ref{l1} applied to $\mathbf{K}_{01}$ and $\mathbf{K}_{02}$ implies
\begin{equation}\label{born14}
\| \mathbf{K}_{0j}(t)f\|_{L^\infty}\le Ct^{-3/2}\|f\|_{L^1},\quad j=1,2,\quad t\ge 1.
\end{equation} 
It remains to estimate $\mathbf{K}_{03}$. Denote $g=\mathcal J_{\frac 12} f$, then we
have $\hat f(k)=(1+k^2)^{-\frac 14}\hat g(k)$. Using formulas
$\mathcal F[\cdot f(\cdot)]=\hat f^\prime$ and
$\hat f'(k)=(1+k^2)^{-\frac 14}\hat g'(k)-\frac k2(1+k^2)^{-\frac 54}\hat g(k)$ we get
\begin{align*}
\int_{\R}[\mathbf{K}_{03}(t)](x,y)f(y)dy&=\frac {\I x}{4\pi t}\sum_{\pm}\int\limits_{|k|\geq 1}\frac{\chi(k)}{k(1+k^2)^{\frac 14}}
\E^{\pm t\sqrt{k^2 + m^2}+\I k x}(1+k^2)^{\frac 14}\hat f(k)dk\\
&+\frac {\I}{4\pi t}\sum_{\pm}\int\limits_{|k|\geq 1}\frac{\chi(k)}{k(1+k^2)^{\frac 14}}
\E^{\pm t\sqrt{k^2 + m^2}+\I k x}\hat g'(k)dk\\
&-\frac {\I}{8\pi t}\sum_{\pm}\int\limits_{|k|\geq 1}\frac{\chi(k)}{(1+k^2)^{\frac 54}}
\E^{\pm t\sqrt{k^2 + m^2}+\I k x}\hat g(k)dk.
\end{align*}
Procceding as in the proof of Lemma \ref{Bessel} we get
\[
\|(1+|\cdot|)^{-1}{\bf K}_{03}(t)f\|_{L^\infty}\leq C t^{-3/2}\Big(\|g\|_{L^1}+\||\cdot|g\|_{L^1}\Big)
\leq C t^{-3/2}\Big(\|f\|_{H^{\frac 12,1}}+\|f\|_{H^{\frac 12,1}_1}\Big),
\]
The last estimate and \eqref{born14} then imply
\[
\Vert \mathbf{K}_0(t)\Vert_{H^{\frac 12,1}_1\to L^{\infty}_{-1}}\le Ct^{-3/2},
\]
Together with \eqref{tmK1} this gives
\[
\Vert[\E^{-\I t\mathbf{H}}]^{12}\xi(\mathbf{H}^2)f\Vert_{L^{\infty}_{-1}}
\le C|t|^{-3/2}\Vert f\Vert_{H^{\frac 12,1}_1}.
\]
Combining this with \eqref{KGn-e} completes the proof of Theorem \ref{Main1} ii).

\appendix
\section{A decay estimate}
\label{A}

The following is \cite[Lem.~6.7]{Cu}) which is an adapted version of \cite[Lemma 2]{MSW}.
We include a proof here for the sake of completeness.

\begin{lemma}[\cite{Cu,MSW}]\label{prop:psi0}
Let $\Lambda(k)$, $k\geq 0$, be a smooth function such that
$\Lambda(k)=O(k^{-5/2})$ as $k\to\infty$ and let 
\[
\Psi(k,t):=\int_{0}^k\E^{\I t\phi(\tau)}d\tau,\quad t\geq 1,\quad k\geq 0,
\]
where $\phi(\tau)=\sqrt{\tau^2 + 1} + v\tau$ with $v\in\R$.
Then the following estimate is valid uniformly with respect to $v$
\begin{equation}\label{newmain}
J(t):= \int_1^t|\Psi(k,t)\Lambda(k)|dk\leq C t^{-1/2}.
\end{equation}
\end{lemma}
\begin{proof}
For simplicity we refer to the van der Corput lemma for the first, second derivative as vdC-1, vdC-2,
respectively (see \cite[Corollary 5 and Lemma 7]{rog}).
First of all, we observe  the second derivative of the phase function $\phi(\tau)$ admits the estimate 
\begin{equation} \label{pprime}
\min\limits_{0\le\tau\le k}\phi^{\prime\prime}(\tau)=\min\limits_{0\le\tau\le k}(1+\tau^2)^{-3/2}=(1+k^2)^{-3/2}.
\end{equation}
Hence, vdC-2 implies
\begin{equation}\label{secondpsi}
 |\Psi(k,t)|\leq Ct^{-1/2}(k+1)^{3/2},\quad k\geq 0,\quad t\geq 1.
\end{equation}
The first derivative of the phase function $\phi(\tau)$ is an increasing function,
satisfying the following estimates 
\begin{equation}\label{phider}
 |\phi^\prime(\tau)|\geq \begin{cases}
  2^{-1/2}, & v\geq 0, \ \tau\geq 1,\\
 \frac 12(\tau^2 +1)^{-1}, & v\leq -1,  \ \tau \geq 0.
\end{cases}
\end{equation}
For $v\in (-1,0)$ the function $\phi^\prime$ has a zero at $\tau_0=-v(1-v^2)^{-1/2}$. 
We study these three regions for $v$ separately. 
For $v\geq 0$ and $k\ge 1$ we use vdC-1  and \eqref{phider} to get the bound
$|\Psi(k,t) - \Psi(1,t)|\leq C t^{-1}$. Since $|\Psi(1,t)|\leq Ct^{-1/2}$ by \eqref{secondpsi} then for $v\geq 0$
\eqref{newmain} follows immediately. Similarly, for $v\leq -1$ from \eqref{phider} and \eqref{secondpsi} it follows
\begin{equation}\label{imp66}
|\Psi(k,t)|\leq C k^2 t^{-1} + |\Psi(1,t)| \le C(k^2 t^{-1} + t^{-1/2}), \quad k\ge 1.
\end{equation}
Thus, \eqref{newmain} holds in this case also.

It remains to consider the case $v\in(-1,0)$, or equivalently $\tau_0\in (0,\infty)$.
In particular, will estimate $J(t)$ in terms of $\tau_0$ rather than $v$.
By monotonicity of $\phi^\prime$ and \eqref{pprime} for all $\tau\in (0,\tau_0/2]$ we get 
\[
|\phi'(\tau)|=
\frac{\tau_0}{\sqrt{\tau_0^2+1}}-\frac{\tau}{\sqrt{\tau^2+1}}\ge\phi'(2\tau)-\phi'(\tau)
\ge\phi''(2\tau)\tau\geq \frac{C}{\tau^2}.
\]
Thus for $\tau_0/2\ge 1$ we obtain similarly as in \eqref{imp66}
\begin{equation}\label{e1}
|\Psi(k,t)|\le C(k^2 t^{-1} + t^{-1/2}),\quad 1\le k\le \tau_0/2.
\end{equation}
Furthermore, for all $\tau\ge 2\tau_0$ we obtain
\[
\phi'(\tau)=
\frac{\tau}{\sqrt{\tau^2+1}}-\frac{\tau_0}{\sqrt{\tau_0^2+1}}\ge \phi'(\tau)-\phi'(\tau/2)
\ge\phi''(\tau)\tau/2\ge\frac{C}{\tau^2}.
\]
Therefore
\begin{equation}\label{e2}
|\Psi(k,t)|\le C(k^2 t^{-1} + t^{-1/2}),\quad \max \{1,2\tau_0\}\le k.
\end{equation}
Moreover, \eqref{secondpsi} implies
\begin{equation}\label{imp77}
\int_{\tau_0/2}^{2\tau_0}|\Psi(k)\Lambda(k)|dk\leq C t^{-1/2},
\end{equation}
since $\int_{y/2}^{2y}k^{-1} dk$ does not depend on $y>0$.
For the same reasons we also have the estimate $J(t)-J(t/4)\le Ct^{-1/2}$.
Moreover in the case $4\le t\le 2\tau_0$ it follows from \eqref{e1} that
$J(t/4)\le Ct^{-1/2}$ and we obtain \eqref {newmain} for $4\le t\le 2\tau_0$.
(Note that in the case  $1\le t\le 4$  \eqref {newmain} holds for any $\tau_0\in (0,\infty)$.)

Next consider $1\le 2\tau_0 \le t$. If, additionally, $\tau_0/2\le 1$, then
\[
J(t)\le\int_{\tau_0/2}^{2\tau_0}|\Psi(k)\Lambda(k)|dk+\int_{2\tau_0}^{t}|\Psi(k)\Lambda(k)|dk\leq C t^{-1/2}
\]
by \eqref{imp77} and \eqref{e2}. In the case $1\le \tau_0/2\le 2\tau_0 \le t$ we get
\[
J(t)\le\int\limits_{1}^{\tau_0/2}|\Psi(k)\Lambda(k)|dk+
\int\limits_{\tau_0/2}^{2\tau_0}|\Psi(k)\Lambda(k)|dk+\int\limits_{2\tau_0}^{t}|\Psi(k)\Lambda(k)|dk\leq C t^{-1/2}
\]
by \eqref{e1}, \eqref{imp77}, and \eqref{e2}.
Finally, in the case $2\tau_0 \le 1$ equation \eqref {newmain} follows from \eqref{e2}.
\end{proof}
\bigskip
\noindent{\bf Acknowledgments.}
We thank Alexander Komech for useful remarks.
I.E.\ gratefully acknowledges the hospitality of the Faculty of Mathematics of the University of Vienna.


\end{document}